\documentclass[11pt]{amsart}

\usepackage{amssymb}
\usepackage{amsfonts}
\usepackage{amsmath}
\usepackage{amsthm}
\usepackage{graphicx}
\usepackage{amsfonts}

\def \IR{\mathbb R}

\newtheorem{ori}{orismos}[section]
\newtheorem{theor}[ori]{Theorem}
\newtheorem{theor1}{Theorem}

\newtheorem{limma}[ori]{Lemma}

\newtheorem{remark}[ori]{Remark}

\begin{document}

\title[INVARIANCE OF GREEN EQUILIBRIUM MEASURE]{INVARIANCE OF GREEN EQUILIBRIUM MEASURE ON THE DOMAIN}

\author{STAMATIS POULIASIS}
\address{Department of  Mathematics\\
Aristotle University of Thessaloniki\\
54124 Thessaloniki, Greece} \email{{\tt spoulias@math.auth.gr}}
\date{28 June 2010}

\maketitle

\begin{abstract}
We prove that the Green equilibrium measure and the
Green equilibrium energy of a compact set $K$
relative to the domains $D$ and $\Omega$ are the same
if and only if $D$ is nearly equal to $\Omega$, for
a wide class of compact sets $K$. Also, we prove that
equality of Green equilibrium measures arises if
and only if the one domain is related with a level
set of the Green equilibrium potential of $K$
relative to the other domain.
\end{abstract}

\section{Introduction}

Consider the following inverse problem in
classical potential theory: Suppose $K$ is
a compact set in a domain $D$. Do the
equilibrium measure and the equilibrium energy
of $K$ relative to $D$ characterize the domain
$D$? More generally, if we know the energy
and the measure on $K$, what can we conclude
about the ambient domain $D$? We proceed to
a rigorous formulation of the problem.

Let $D$ be a Greenian open subset of $\IR^{n}$ and
denote by $G_{D}(x,y)$ the Green function of $D$. Also
let $K$ be a compact subset of $D$. The \textit{Green equilibrium
energy} of $K$ relative to $D$ is defined by
$$I(K,D)=\inf_{\mu}\int\!\!\!\!\int G_{D}(x,y)d\mu(x)d\mu(y),$$
where the infimum is taken over all unit Borel measures
$\mu$ supported on $K$. The \textit{Green capacity} of $K$
relative to $D$ is the number
$$C_{D}(K)=\frac{1}{I(K,D)}.$$
When $I(K,D)<+\infty$, the unique unit Borel measure
$\mu_{K}$ for which the above infimum is attained
is the \textit{Green equilibrium measure} and
the function
$$U_{\mu_{K}}^{D}(x)=\int G_{D}(x,y)d\mu_{K}(y)$$
is the \textit{Green equilibrium potential} of $K$ relative
to $D$.
See e.g. \cite[p. 174]{3} or \cite[p. 134]{1}.  We denote by
$C_{2}(E)$ the logarithmic ($n=2$) or Newtonian ($n\geq3$)
capacity of the Borel set $E$. When $E\subset D$, the equalities $C_{D}(E)=0$
and $C_{2}(E)=0$ are equivalent; \cite[p. 174]{3}.
If two Borel sets $A,B\subset\IR^{n}$ differ only on a set of
zero capacity (namely,
$C_{2}(A\setminus B)=C_{2}(B\setminus A)=0$), then we
say that $A,B$ are \textit{nearly everywhere equal} and
write $A{\stackrel{n.e.}{=}}B$. Nearly everywhere
equal sets have the same potential theoretic
behavior.

\medskip

Suppose that $D$ is a domain and $E$ is a compact subset of $D\setminus K$.
If $C_{2}(E)=0$, then the Green capacity
and the Green equilibrium measure of $K$ relative to the
open sets $D$ and $D\setminus E$ are the same. That
happens because the sets $D$ and $D\setminus E$
are nearly everywhere equal.
The following question arises naturally:

\begin{quote}
Does there exist a Greenian domain $\Omega$,
not nearly everywhere equal to $D$, such that
the Green capacity
and the Green equilibrium measure of $K$ relative to
$D$ and $\Omega$ are the same?
\end{quote}

In general the answer is positive. However, we
shall show that for
a large class of compact sets (for example for
compact sets with nonempty
interior) the answer is negative. That is,
given a compact set $K$ with nonempty interior,
the Green equilibrium measure $\mu_{K}$ and the
positive number $C_{D}(K)$ completely characterizes
$D$ from the point of view of potential theory.

There is a second question:

\begin{quote}
Suppose that the Green equilibrium measures
of $K$ relative to $D$ and $\Omega$ are the same.
How are the sets $D$ and $\Omega$ related?
\end{quote}
In that case, the boundary of the domain relative
to which $K$ has smaller Green energy is
a level set of the Green equilibrium potential
of $K$ relative to the other domain.

\medskip

In our main result we give an answer
to the above questions, for compact sets with
nonempty interior.

\begin{theor1}\label{idiometrogreen}
Let $K$ be a compact subset of $\IR^{n}$ with nonempty interior. Let $D_{1},D_{2}$
be two Greenian subdomains of $\IR^{n}$ that contain $K$ and
let $\mu_{1}$ and $\mu_{2}$ be the Green equilibrium
measures of $K$ relative to $D_{1}$ and $D_{2}$,
respectively. Then\smallskip

$\mbox{ }({\rm{i}})$ $\mu_{1}=\mu_{2}$ and $I(K,D_{1})=I(K,D_{2})$
if and only if $D_{1}{\stackrel{n.e.}{=}}D_{2}$.

\smallskip

$({\rm{ii}})$ If $I(K,D_{1})<I(K,D_{2})$ and the set
$$\tilde{D_{2}}=\{x\in D_{2}:U_{\mu_{2}}^{D_{2}}(x)>I(K,D_{2})-I(K,D_{1})\}$$
contains $K$, we have
$$ \mu_{1}=\mu_{2} \mbox{ if and only if } D_{1}{\stackrel{n.e.}{=}}\tilde{D_{2}}.$$
\end{theor1}

Moreover, we shall show that Theorem \ref{idiometrogreen} is
valid for a much wider class of compact sets (see Remark \ref{extension}).

\medskip

In the following section we introduce the
concepts of Green potential theory that are needed for
our results. Theorem \ref{idiometrogreen} is proved
in section 3. In section 4 we examine the case of compact
sets $K$ with empty interior, we give some
counterexamples and we pose a conjecture and a question.

\section{Background Material}

We denote by $B(x,r)$ and $S(x,r)$ the open ball and the
sphere with center
$x$ and radius $r$ in $\IR^{n}$, respectively. For a set $E\subset\IR^{n}$,
the interior, the closure and the boundary of $E$ are
denoted by $E^{\circ}$, $\overline{E}$ and $\partial E$,
respectively. The surface area of the unit sphere
$S(0,1)$ of $\IR^{n}$ is denoted by $\sigma_{n}$.

If a property holds for all the points of a set
$A$ apart from a set of zero capacity we will say that
the property holds for \textit{nearly every} ($n.e.$) point of $A$.

The logarithmic ($n=2$) or Newtonian ($n\geq3$) kernel
is denoted by
\begin{equation}
\mathcal{K}(x,y) := \left\{ \begin{array}{cc}
                       \log\frac{1}{|x-y|}, & n=2, \\
                       \frac{1}{|x-y|^{n-2}}, & n\geq 3, \nonumber
                       \end{array}
               \right.
\end{equation}
for $(x,y)\in\IR^{n}\times\IR^{n}$. If $D\subset\IR^{n}$
is a Greenian open set then
$$G_{D}(x,y)=\mathcal{K}(x,y)+H(x,y)$$
where, for fixed $y$, $H(\cdot,y)$ is the greatest
harmonic minorant of $\mathcal{K}(\cdot,y)$ on $D$. If $\mu$
is a measure with compact support on $D$, the Green
potential of $\mu$ is the function
$$U_{\mu}^{D}(x)=\int G_{D}(x,y)d\mu(y),\qquad x\in D.$$
Then (see \cite[pp. 96-104]{1}) $U_{\mu}^{D}$ is superharmonic on $D$, harmonic on
$D\setminus {\rm{supp}}(\mu)$ and the Riesz measure of $U_{\mu}^{D}$
is proportional to $\mu$; more precisely, $\Delta U_{\mu}^{D}=-\kappa_{n}\mu$,
where $\Delta$ is the distributional Laplacian on $D$ and
\begin{equation}
\kappa_{n} = \left\{ \begin{array}{cc}
                       \sigma_{2}, & n=2, \\
                       (n-2)\sigma_{n}, & n\geq 3. \nonumber
                       \end{array}
               \right.
\end{equation}
We shall need the following result for the boundary
behavior of a Green potential.
\begin{theor}\label{boundarybehavior}\cite[p. 148]{1}
Let $\Omega$ be a Greenian open subset of $\IR^{n}$
and let $\mu$ be a Borel measure with compact
support in $\Omega$. Then
$$\lim_{\Omega\ni x\to\xi}U_{\mu}^{\Omega}(x)=0,$$
for nearly every point $\xi\in\partial\Omega$.
\end{theor}
Also, if $U_{\mu_{K}}^{D}$ is the Green equilibrium potential of a
compact set $K$ relative to $D$, then the equality
\begin{equation}\label{equilibpotentialbehavior}
U_{\mu_{K}}^{D}(x)=I(K,D)
\end{equation}
holds for all $x\in K^{\circ}$ and for nearly every point $x\in K$;
(\cite[p. 138]{2}).

The following theorem gives a geometric
interpretation of the fact that sets of zero capacity
are negligible.

\begin{theor}\label{connectedness}{{\rm{(}}\cite[p. 125]{1}{\rm{)}}.}
Let $\Omega\subset\IR^{n}$ be a domain and $E$
a relatively closed subset of $\Omega$ with
$C_{2}(E)=0$. Then the set $\Omega\setminus E$ is
connected.
\end{theor}

\begin{remark}\label{condensers}
{\rm{When the open set $D\setminus K$ is connected, it is called
a \textit{condenser} and the sets $\IR^{n}\setminus D$
and $K$ are called the \textit{plates} of the condenser; see \cite{5,6}.
It is well known that
Green potential theory and condenser theory are
equivalent; see e.g. \cite[p. 700-701]{7} or \cite[p. 393]{8}. Therefore,
Theorem \ref{idiometrogreen} can be restated
as a theorem about condensers.}}
\end{remark}

\section{Dependence of Green equilibrium measure on the domain}

Let $D$ be a Greenian open subset of $\IR^{n}$ and
$K$ a compact subset of $D$. First we shall examine
the open sets bounded by the level surfaces of the
Green equilibrium potential of $K$ relative to $D$.
The Green equilibrium measure is the same for each of these open sets.

\begin{limma}\label{leveldomain}
Let $D$ be a Greenian open subset of $\IR^{n}$ and let
$K$ be a compact subset of $D$ that has finite
Green equilibrium energy relative to $D$.
Also let $\mu_{K}$ be the Green equilibrium
measure of $K$ relative to $D$ and for $0<\alpha<I(K,D)$,
let
$$D_{\alpha}=\{x\in D:U_{\mu_{K}}^{D}(x)>\alpha\}.$$
If $K\subset D_{\alpha}$ and $\mu_{\alpha}$ is the
Green equilibrium measure of $K$ relative to $D_{\alpha}$,
then $\mu_{\alpha}=\mu_{K}$ and
$$I(K,D_{\alpha})=I(K,D)-\alpha.$$
\end{limma}

\begin{proof}
The Green equilibrium potential $U_{\mu_{K}}^{D}$ is
a superharmonic function on $D$ so, in particular, it
is lower semicontinuous on $D$. Therefore the sets
$$D_{\alpha}=\{x\in D:U_{\mu_{K}}^{D}(x)>\alpha\}$$
are open for all $\alpha\in(0,I(K,D))$.

Choose $\alpha\in(0,I(K,D))$ such that $K\subset D_{\alpha}$.
Consider the function $U:\overline{D_{\alpha}}\mapsto\IR$
with
$$U(x)=U_{\mu_{K}}^{D}(x)-\alpha.$$
Then $U$ is harmonic on $D_{\alpha}\setminus K$. Also,
by property (\ref{equilibpotentialbehavior}) of the
Green equilibrium potentials, $U$ has boundary values
($I(K,D)-\alpha$) n.e. on $\partial K$ and $0$ on $\partial D_{\alpha}$.
Therefore, by the extended maximum principle and the boundary behavior
of the Green equilibrium potential $U_{\mu_{\alpha}}^{D_{\alpha}}$
on $D_{\alpha}\setminus K$,
$$U_{\mu_{\alpha}}^{D_{\alpha}}(x)=\frac{I(K,D_{\alpha})}{I(K,D)-\alpha}U(x).$$
Applying the distributional Laplacian on $D_{\alpha}$ we get
$$\Delta(U_{\mu_{\alpha}}^{D_{\alpha}})=
\Delta\Big(\frac{I(K,D_{\alpha})}{I(K,D)-\alpha}(U_{\mu_{K}}^{D}-\alpha)\Big),$$
so
$$-\kappa_{n}\mu_{\alpha}=-\kappa_{n}\frac{I(K,D_{\alpha})}{I(K,D)-\alpha}\mu_{K}$$
and
$$\mu_{\alpha}=\frac{I(K,D_{\alpha})}{I(K,D)-\alpha}\mu_{K}.$$
Since $\mu_{\alpha}$ and $\mu_{K}$ are unit measures, we
conclude that $\mu_{\alpha}=\mu_{K}$ and
$I(K,D_{\alpha})=I(K,D)-\alpha$.
\end{proof}

In order to proof Theorem \ref{idiometrogreen} we shall
need some more lemmas. In the following lemma we
show that the difference of two Green potentials of the same
measure is a harmonic function.

\begin{limma}\label{harmonicdifference}
Let $D_{1}$ and $D_{2}$ be two Greenian open subsets of $\IR^{n}$
such that $D_{1}\cap D_{2}\neq\varnothing$. Also, let $\mu$
be a Borel measure with compact support in $D_{1}\cap D_{2}$
such that the Green potentials $U_{\mu}^{D_{1}}$ and
$U_{\mu}^{D_{2}}$ are finite on $D_{1}\cap D_{2}$. Then
the difference
$$U_{\mu}^{D_{1}}-U_{\mu}^{D_{2}}$$
is a harmonic function on $D_{1}\cap D_{2}$.
\end{limma}

\begin{proof}
Let
$$G_{i}(x,y)=\mathcal{K}(x,y)+H_{i}(x,y)$$
be the Green function of $D_{i}$,
$i=1,2$, respectively. Then
\begin{eqnarray*}
  U_{\mu}^{D_{1}}(x)-U_{\mu}^{D_{2}}(x) &=& \int_{K}\mathcal{K}(x,y)d\mu(y)+\int_{K}H_{1}(x,y)d\mu(y) \\
   && -\int_{K}\mathcal{K}(x,y)d\mu(y)-\int_{K}H_{2}(x,y)d\mu(y) \\
   &=& \int_{K}[H_{1}(x,y)-H_{2}(x,y)]d\mu(y).
\end{eqnarray*}
The difference $H_{1}(x,y)-H_{2}(x,y)$ is a harmonic
function on $D_{1}\cap D_{2}$ on both variables $x$
and $y$, separately. Therefore,
$$x\mapsto U_{\mu}^{D_{1}}(x)-U_{\mu}^{D_{2}}(x)
=\int_{K}[H_{1}(x,y)-H_{2}(x,y)]d\mu(y)$$
is a harmonic function on $D_{1}\cap D_{2}$;
see (\cite[Lemma 6.7, p. 103]{2}).
\end{proof}

We shall need the fact that particular parts of
the boundary of an open set belong to the reduced
kernel \cite[p. 164]{3} of the boundary.

\begin{limma}\label{capacitydensityboundary}
Let $\Omega$ be an open subset of $\IR^{n}$.
Suppose that $\IR^{n}\setminus\overline{\Omega}\neq\varnothing$
and let $O$ be a connected component
of $\IR^{n}\setminus\overline{\Omega}$. Then for all $\xi\in\partial O$
and for all $r>0$,
$$C_{2}(B(\xi,r)\cap\partial O)>0.$$
\end{limma}

\begin{proof}
Let $\xi\in\partial O$ and $r>0$. Then $B(\xi,r)\setminus\partial O$
intersects $\Omega$ and $O$, so it is not connected.
Then, since $B(\xi,r)\cap\partial O$ is
a relatively closed subset of $B(\xi,r)$, it follows from
Theorem \ref{connectedness} that
$$C_{2}(B(\xi,r)\cap\partial O)>0.$$
\end{proof}

The next lemma gives a characterization of
nearly everywhere equal sets: the Green
potential of any measure with compact support
is the same.

\begin{limma}\label{nearlysame}
Let $D_{1}$ and $D_{2}$ be two Greenian domains
of $\IR^{n}$.
Then $D_{1}{\stackrel{n.e.}{=}}D_{2}$
if and only if there exists a Borel measure $\mu$ with
compact support in $D_{1}\cap D_{2}$ such that
$U_{\mu}^{D_{1}},U_{\mu}^{D_{2}}$ are non constant
functions on each connected component of $D_{1}\cap D_{2}$
and $U_{\mu}^{D_{1}}=U_{\mu}^{D_{2}}$
on an open ball $B\subset D_{1}\cap D_{2}$.
\end{limma}

\begin{proof} Suppose that
$D_{1}{\stackrel{n.e.}{=}}D_{2}$. Then
$G_{D_{1}}(x,y)=G_{D_{2}}(x,y)$
for all $x,y\in D_{1}\cap D_{2}$
(\cite[Corollary 5.2.5, p. 128]{1}). Therefore
$$U_{\mu}^{D_{1}}(x)=\int G_{D_{1}}(x,y)d\mu(y)=
\int G_{D_{2}}(x,y)d\mu(y)=U_{\mu}^{D_{2}}(x),$$
for all $x\in D_{1}\cap D_{2}$ and for all
Borel measures $\mu$ with
compact support in $D_{1}\cap D_{2}$.

\medskip

Conversely, suppose that $\mu$ is a Borel measure with
compact support in $D_{1}\cap D_{2}$ such that
$U_{\mu}^{D_{1}},U_{\mu}^{D_{2}}$ are non constant
functions on each connected component of $D_{1}\cap D_{2}$
and $U_{\mu}^{D_{1}}=U_{\mu}^{D_{2}}$
on a ball $B\subset D_{1}\cap D_{2}$.
By Lemma \ref{harmonicdifference}, the function
$$u=U_{\mu}^{D_{1}}-U_{\mu}^{D_{2}}$$
is harmonic on $D_{1}\cap D_{2}$
and vanishes on the ball $B$.
Therefore, by the identity principle
(\cite[Lemma 1.8.3, p. 27]{1}),
$u$ vanishes on the connected component $A$
of $D_{1}\cap D_{2}$ that contains $B$. That
is, $U_{\mu}^{D_{1}}=U_{\mu}^{D_{2}}$ on $A$.

Suppose that $C_{2}(D_{1}\setminus A)>0$. We shall show
that $C_{2}(D_{1}\cap\partial A)>0$. Consider the decomposition
$$D_{1}\setminus A=(D_{1}\cap\partial A)\cup
(D_{1}\cap(\IR^{n}\setminus \overline{A})).$$
If $D_{1}\cap(\IR^{n}\setminus \overline{A})=\varnothing$,
then
$$C_{2}(D_{1}\cap\partial A)=C_{2}(D_{1}\setminus A)>0.$$
Suppose that $D_{1}\cap(\IR^{n}\setminus \overline{A})\neq\varnothing$.
Let $O$ be a connected
component of $\IR^{n}\setminus \overline{A}$ that intersects $D_{1}$.
Since $D_{1}$ intersects $O$ and $\IR^{n}\setminus O$, $D_{1}\cap\partial O\neq\varnothing$.
Let $\xi\in D_{1}\cap\partial O$ and $r>0$ such that $B(\xi,r)\subset D_{1}$.
By Lemma \ref{capacitydensityboundary},
$$C_{2}(D_{1}\cap\partial A)\geq C_{2}(B(\xi,r)\cap\partial O)>0.$$
Therefore, in any case $C_{2}(D_{1}\cap\partial A)>0$. Then,
since $\partial A\subset(\partial D_{1}\cup\partial D_{2})$
and $D_{1}\cap\partial A\subset D_{1}$, we have that
$D_{1}\cap\partial A$ is a subset of $\partial D_{2}$
with positive capacity. From Theorem \ref{boundarybehavior}
we have that there exist $\xi_{0}\in D_{1}\cap\partial A$
such that
$$\lim_{A\ni x\to\xi_{0}}U_{\mu}^{D_{2}}(x)=
\lim_{D_{2}\ni x\to\xi_{0}}U_{\mu}^{D_{2}}(x)=0.$$
Since ${\rm{supp}}(\mu)$ is a compact subset of $D_{1}\cap D_{2}$,
$\xi_{0}\in D_{1}$ and $\xi_{0}\in\partial D_{2}$, $U_{\mu}^{D_{1}}$
is harmonic on an open neighborhood of $\xi_{0}$. So
$$\lim_{A\ni x\to\xi_{0}}U_{\mu}^{D_{1}}(x)=
\lim_{D_{1}\ni x\to\xi_{0}}U_{\mu}^{D_{1}}(x)=
U_{\mu}^{D_{1}}(\xi_{0})>0.$$
Then, since $U_{\mu}^{D_{1}}=U_{\mu}^{D_{2}}$ on $A$,
we obtain
$$0=\lim_{A\ni x\to\xi_{0}}U_{\mu}^{D_{2}}(x)=
\lim_{A\ni x\to\xi_{0}}U_{\mu}^{D_{1}}(x)>0,$$
which is a contradiction.

Therefore $C_{2}(D_{1}\setminus A)=0$. Since
$A\subset D_{1}$, we get $D_{1}{\stackrel{n.e.}{=}}A$.
In a similar way we can show that
$D_{2}{\stackrel{n.e.}{=}}A$. Therefore, in
particular, $D_{1}{\stackrel{n.e.}{=}}D_{2}$.
\end{proof}

We proceed to prove Theorem \ref{idiometrogreen}.

\medskip

{\textbf{Proof of Theorem \ref{idiometrogreen}.}} (i) Suppose that
$D_{1}{\stackrel{n.e.}{=}}D_{2}$. Then the
relations $\mu_{1}=\mu_{2}$ and $I(K,D_{1})=I(K,D_{2})$
follow from the equality
$$G_{D_{1}}(x,y)=G_{D_{2}}(x,y)$$
for all $x,y\in K$
(\cite[Corollary 5.2.5, p. 128]{1}).

\medskip

Conversely, suppose that $\mu_{1}=\mu_{2}$ and $I(K,D_{1})=I(K,D_{2})$.
By property (\ref{equilibpotentialbehavior}) of the
Green equilibrium potentials,
$$U_{\mu_{1}}^{D_{1}}(x)-U_{\mu_{2}}^{D_{2}}(x)=
I(K,D_{1})-I(K,D_{2})=0,$$
for all $x\in K^{\circ}$. Then, by Lemma \ref{nearlysame},
$D_{1}{\stackrel{n.e.}{=}}D_{2}$.

\bigskip

(ii) Let $\tilde{\mu_{2}}$ be the
Green equilibrium measure of $K$ relative to $\tilde{D_{2}}$.
By Lemma
\ref{leveldomain}, $\tilde{\mu_{2}}=\mu_{2}$ and
$I(K,\tilde{D_{2}})=I(K,D_{1})$. Suppose that
$D_{1}{\stackrel{n.e.}{=}}\tilde{D_{2}}$.
Since $D_{1}$ is a domain, by
Theorem \ref{connectedness}, $\tilde{D_{2}}$ is also a domain.
Then by (i), $\mu_{1}=\tilde{\mu_{2}}=\mu_{2}$.

\medskip

Conversely, suppose that $\mu_{1}=\mu_{2}$. Let $O$ be a connected component
of $\tilde{D_{2}}$ that intersects the interior of $K$.
From property (\ref{equilibpotentialbehavior}) of the
Green equilibrium potentials,
$$U_{\mu_{1}}^{D_{1}}(x)-U_{\tilde{\mu_{2}}}^{\tilde{D_{2}}}(x)=
I(K,D_{1})-I(K,\tilde{D_{2}})=0,$$
for all $x\in K^{\circ}\cap O$.
Also, $U_{\tilde{\mu_{2}}}^{\tilde{D_{2}}}=U_{\tilde{\mu_{2}}|_{O}}^{O}$ on $O$
and $\mu_{1}=\mu_{2}=\tilde{\mu_{2}}$.
By Lemma \ref{nearlysame},
$D_{1}{\stackrel{n.e.}{=}}O$.
Therefore
$$D_{1}\cap (\tilde{D_{2}}\setminus O)=\varnothing$$
and $K\subset O$, since $K\subset\tilde{D_{2}}$. Suppose that $\tilde{D_{2}}\neq O$
and let $A$ be a second connected component of $\tilde{D_{2}}$.
Then $K\cap A=\varnothing$, $U_{\mu_{2}}^{D_{2}}$ is harmonic
on $A$ and
$$U_{\mu_{2}}^{D_{2}}=I(K,D_{2})-I(K,D_{1})$$
on $\partial A$. By the maximum principle,
$U_{\mu_{2}}^{D_{2}}=I(K,D_{2})-I(K,D_{1})$ on $A$ and
by the identity principle
$U_{\mu_{2}}^{D_{2}}=I(K,D_{2})-I(K,D_{1})$ on the connected
component $B$ of $D_{2}\setminus K$
that contains $A$. Also $\partial B\subset(\partial D_{2}\cup\partial K)$.
But
$$\lim_{D_{2}\ni x\to\zeta}U_{\mu_{2}}^{D_{2}}(x)=0\neq I(K,D_{2})-I(K,D_{1}),
\mbox{ for n.e. } \zeta\in\partial D_{2}$$
and
$$U_{\mu_{2}}^{D_{2}}=I(K,D_{2})\neq I(K,D_{2})-I(K,D_{1}),\mbox{ n.e. on }\partial K,$$
which is a
contradiction. Therefore $\tilde{D_{2}}=O$ and $D_{1}{\stackrel{n.e.}{=}}\tilde{D_{2}}$.

\begin{remark}
{\rm{Let $B$ be a ball that contains $K$. If the open set $B\setminus K$ is regular
for the Dirichlet problem, the property
$K\subset\tilde{D_{2}}$ is always true.}}
\end{remark}

\section{Compact sets with empty interior and counterexamples}

Let $K$ be a compact subset of $\IR^{n}$ and
let $D_{1},D_{2}$ be two Greenian subdomains of
$\IR^{n}$ that contain $K$. We shall examine
the case where the domains $D_{1},D_{2}$ are
not nearly everywhere equal and the Green
equilibrium measures of $K$ relative to
$D_{1},D_{2}$ are the same. It turns out that,
in most cases, $K$ must be sufficiently smooth.

\begin{theor}\label{emptyinterior}
Let $K$ be a compact subset of $\IR^{n}$ with positive
capacity and
let $D_{1},D_{2}$ be two Greenian subdomains of
$\IR^{n}$ that are not nearly
everywhere equal and contain $K$.
Also let $\mu_{1}$ and $\mu_{2}$ be the Green equilibrium
measures of $K$ relative to $D_{1}$ and $D_{2}$,
respectively. If $\mu_{1}=\mu_{2}$ and $I(K,D_{1})=I(K,D_{2})$, then there exists a
level set $L$ of a non constant harmonic function
such that $C_{2}(K\setminus L)=0$.
\end{theor}

\begin{proof}
By Lemma \ref{harmonicdifference}, the function
$$h(x)=U_{\mu_{1}}^{D_{1}}(x)-U_{\mu_{2}}^{D_{2}}(x)$$
is harmonic on $D_{1}\cap D_{2}$. Since $D_{1},D_{2}$
are not nearly everywhere equal, Lemma
\ref{nearlysame} shows that $h$ cannot be $0$ on a non-empty open
subset of $D_{1}\cap D_{2}$. From property (\ref{equilibpotentialbehavior}) of the
Green equilibrium potentials,
$$h(x)=U_{\mu_{1}}^{D_{1}}(x)-U_{\mu_{2}}^{D_{2}}(x)=
I(K,D_{1})-I(K,D_{2})=0,$$
for nearly every point $x\in K$.
So, $h$ is a
non constant harmonic function on every connected component $A$ of $D_{1}\cap D_{2}$
such that $C_{2}(A\cap K)>0$.
Let $G$ be the union of all the connected components $A$ of $D_{1}\cap D_{2}$
such that $C_{2}(A\cap K)>0$ and let
$$L=\{x\in G:h(x)=0\}.$$
Then $h$ is a
non constant harmonic function on $G$, $L$ is a level set of $h$ and
$C_{2}(K\setminus L)=0$.
\end{proof}

\begin{remark}\label{extension}
{\rm{It follows from Theorem \ref{emptyinterior} that Theorem
\ref{idiometrogreen} is valid for all
compact subsets $K$ that cannot be contained,
except on a set of zero capacity, on a level
set of a non constant harmonic function.
An example of a compact set in the above
class is every $(n-1)$-dimensional compact
submanifold of $\IR^{n}$ that is not real analytic
on a relatively open subset. Another example
is every compact set $K\subset\IR^{2}$
which has at least one connected component
that is neither a
singleton nor a piecewise analytic arc.}}
\end{remark}

\medskip

We proceed to give examples of compact
sets $K$ and pairs of domains that are
answers to the first question we posed
in the introduction. Keeping in mind
Remark \ref{extension}, the compact
sets will lie on sufficiently smooth
subsets of $\IR^{n}$, mainly spheres
and hyperplanes.

\medskip

Let $K$ be a compact subset of a sphere $S(x_{0},r)$
with $C_{2}(K)>0$. Let $D$ be a Greenian open subset of
$\IR^{n}$ that contains $K$. We denote by $D^{*}$ the
inverse of $D$ with respect to $S(x_{0},r)$; see e.g.
\cite[p. 19]{1}. Then $D^{*}$ is Greenian and (see e.g.
\cite[p. 95]{1})
$$G_{D^{*}}(x,y)=\Big(\frac{r^{2}}{|x-x_{0}|\mbox{ }
|y-x_{0}|}\Big)^{n-2}G_{D}(x^{*},y^{*}),\qquad x,y\in D^{*}.$$
Moreover, $G_{D}(x,y)=G_{D^{*}}(x,y)$ for every $x,y\in K$.
So $I(K,D)=I(K,D^{*})$ and the Green equilibrium
measures of $K$ relative to $D$ and $D^{*}$ are the
same. Of course, the sets $D$ and $D^{*}$ are not
nearly everywhere equal in general. A similar result
holds also in the case when $K$ is a subset of a hyperplane $H$ and
$D^{*}$ is the reflection of $D$ with respect to $H$.

\bigskip

Finally we state a conjecture and a question:

\medskip

\textbf{Conjecture}: Let $K$ be a subset of a sphere $S(x_{0},r)$
with $C_{2}(K)>0$ and let $D,\Omega$ be two Greenian domains
that contain $K$. If $I(K,D)=I(K,\Omega)$ and the Green equilibrium
measures of $K$ relative to $D$ and $\Omega$ are the
same, then $\Omega{\stackrel{n.e.}{=}}D$ or
$\Omega{\stackrel{n.e.}{=}}D^{*}$ where $D^{*}$ is the
inverse of $D$ with respect to $S(x_{0},r)$.

\medskip

\textbf{Question}: Let $n\geq3$. Let $K$ be a
compact subset of $\IR^{n}$ and suppose that
there does not exist a sphere $S$ or a
hyperplane $H$ such that $C_{2}(K\setminus S)=0$
or $C_{2}(K\setminus H)=0$. Do there exist
domains $D,\Omega$
which are not nearly everywhere equal
and contain $K$ such that $I(K,D)=I(K,\Omega)$ and the Green equilibrium
measures of $K$ relative to $D$ and $\Omega$ are the
same?


\end{document}